\documentclass[11pt,reqno]{amsart}

\usepackage{amscd,amssymb,amsmath,amsthm}
\usepackage[arrow,matrix]{xy}
\usepackage{graphicx}
\usepackage{epstopdf}
\usepackage{color}
\newtheorem{thm}{Theorem}

\newtheorem{lemma}{Lemma}
\newtheorem{pro}{Proposition}

\newtheorem{cor}{Corollary}

\numberwithin{equation}{section} 

\def\r{\rho}

\def\C{\mathbb C}

\def\C{\mathbb{C}}

\begin{document}

\title[$p$-adic dynamical systems of a $(1,2)$-rational function]{$p$-adic dynamical systems of the function $\dfrac{ax}{x^2+a}$}

\author{U.A. Rozikov, I.A. Sattarov, S. Yam}

 \address{U.\ A.\ Rozikov \\ Institute of mathematics,
81, Mirzo Ulug'bek str., 100125, Tashkent, Uzbekistan.} \email
{rozikovu@yandex.ru}

 \address{I.\ A.\ Sattarov \\ Institute of mathematics,
81, Mirzo Ulug'bek str., 100125, Tashkent, Uzbekistan.} \email
{sattarovi-a@yandex.ru}

\address{S.\ Yam \\ California State University, Monterey Bay, 100 Campus Center, Seaside, California, 93955 USA} \email
{syam@csumb.edu}

\begin{abstract} We show that any $(1,2)$-rational function with a unique
fixed point is topologically conjugate to a $(2,2)$-rational function
or to the function $f(x)={ax\over x^2+a}$. The case $(2,2)$ was studied in our previous paper,
here we study  the dynamical systems generated by the function $f$ on the set of complex
$p$-adic field $\C_p$. We show that the unique fixed point is indifferent and therefore
the convergence of the trajectories is not the typical
case for the dynamical systems. We construct the corresponding Siegel disk of these dynamical systems.
 We determine a sufficiently small set containing the set of limit points.
It is given all possible invariant spheres.
 We show that the $p$-adic dynamical system reduced on each invariant sphere is not ergodic with respect to Haar measure on the set of $p$-adic numbers $Q_p$.
 Moreover some periodic orbits of the system are investigated.
\end{abstract}

 \maketitle

\section{Introduction}

We study $p$-adic dynamical systems generated by a rational
function. For motivation of such investigations see \cite{ARS}-\cite{S} and references therein.
The paper is organized as follows: First we give some necessary definitions and facts.
Then in  Section 2 show that any $(1,2)$-rational function with a unique
fixed point is topologically conjugate to a $(2,2)$-rational function
or to the function $f(x)={ax\over x^2+a}$. In \cite{RS2} the case of $(2,2)$-rational function
with a unique fixed point is studied.  In this paper for $f$ we show that the unique fixed point is indifferent.
 We give a Siegel disk of the dynamical system. We give a localization
 of the set of limit points.
Section 3 contains a description of all invariant spheres with respect to $f$. We study ergodicity properties of the
dynamical system reduced on each invariant sphere with respect to Haar measure and
 show that the $p$-adic dynamical system reduced on each invariant sphere is not ergodic.
In Section 4 we find 2-periodic orbit $\{t_1, t_2\}$ and show that it can only be either an attracting or an indifferent. We
shall prove that if the cycle is attracting then it attracts each trajectory which starts from an element of a open ball of radius
$h=|t_1 - t_2|_p$ centered at $t_1$ or at $t_2$. If the 2-periodic cycle is an indifferent one then every iteration maps either of the two
aforementioned balls to another one. All other spheres of radius $> h$ and center $t_1$ and $t_2$ are invariant independently of the
attractiveness of the cycle.

\subsection{$p$-adic numbers}

Let $\mathbb{Q}$ be the field of rational numbers. The greatest common
divisor of the positive integers $n$ and $m$ is denotes by
$(n,m)$. Every rational number $x\neq 0$ can be represented in the
form $x=p^r\frac{n}{m}$, where $r,n\in\mathbb{Z}$, $m$ is a
positive integer, $(p,n)=1$, $(p,m)=1$ and $p$ is a fixed prime
number.

The $p$-adic norm of $x$ is given by
$$
|x|_p=\left\{
\begin{array}{ll}
p^{-r}, & \ \textrm{ for $x\neq 0$},\\[2mm]
0, &\ \textrm{ for $x=0$}.\\
\end{array}
\right.
$$
It has the following properties:

1) $|x|_p\geq 0$ and $|x|_p=0$ if and only if $x=0$,

2) $|xy|_p=|x|_p|y|_p$,

3) the strong triangle inequality
$$
|x+y|_p\leq\max\{|x|_p,|y|_p\},
$$

3.1) if $|x|_p\neq |y|_p$ then $|x+y|_p=\max\{|x|_p,|y|_p\}$,

3.2) if $|x|_p=|y|_p$ then $|x+y|_p\leq |x|_p$,

this is a non-Archimedean one.

The completion of $\mathbb{Q}$ with  respect to $p$-adic norm defines the
$p$-adic field which is denoted by $\mathbb{Q}_p$ (see \cite{Ko}).

The algebraic completion of $\mathbb{Q}_p$ is denoted by $\C_p$ and it is
called {\it complex $p$-adic numbers}.  For any $a\in\C_p$ and
$r>0$ denote
$$
U_r(a)=\{x\in\C_p : |x-a|_p<r\},\ \ V_r(a)=\{x\in\C_p :
|x-a|_p\leq r\},
$$
$$
S_r(a)=\{x\in\C_p : |x-a|_p= r\}.
$$

A function $f:U_r(a)\to\C_p$ is said to be {\it analytic} if it
can be represented by
$$
f(x)=\sum_{n=0}^{\infty}f_n(x-a)^n, \ \ \ f_n\in \C_p,
$$ which converges uniformly on the ball $U_r(a)$.

\subsection{Dynamical systems in $\C_p$}

Recall some known facts concerning dynamical
systems $(f,U)$ in $\C_p$, where $f: x\in U\to f(x)\in U$ is an
analytic function and $U=U_r(a)$ or $\C_p$ (see for example \cite{PJS}).

Now let $f:U\to U$ be an analytic function. Denote
$f^n(x)=\underbrace{f\circ\dots\circ f}_n(x)$.

If $f(x_0)=x_0$ then $x_0$
is called a {\it fixed point}. The set of all fixed points of $f$
is denoted by Fix$(f)$. A fixed point $x_0$ is called an {\it
attractor} if there exists a neighborhood $U(x_0)$ of $x_0$ such
that for all points $x\in U(x_0)$ it holds
$\lim\limits_{n\to\infty}f^n(x)=x_0$. If $x_0$ is an attractor
then its {\it basin of attraction} is
$$
A(x_0)=\{x\in \C_p :\ f^n(x)\to x_0, \ n\to\infty\}.
$$
A fixed point $x_0$ is called {\it repeller} if there  exists a
neighborhood $U(x_0)$ of $x_0$ such that $|f(x)-x_0|_p>|x-x_0|_p$
for $x\in U(x_0)$, $x\neq x_0$.

Let $x_0$ be a fixed point of a
function $f(x)$.
Put $\lambda=f'(x_0)$. The point $x_0$ is attractive if $0<|\lambda|_p < 1$, {\it indifferent} if $|\lambda|_p = 1$,
and repelling if $|\lambda|_p > 1$.

The ball $U_r(x_0)$ (contained in $V$) is said to
be a {\it Siegel disk} if each sphere $S_{\r}(x_0)$, $\r<r$ is an
invariant sphere of $f(x)$, i.e. if $x\in S_{\r}(x_0)$ then all
iterated points $f^n(x)\in S_{\r}(x_0)$ for all $n=1,2\dots$.  The
union of all Siegel desks with the center at $x_0$ is said to {\it
a maximum Siegel disk} and is denoted by $SI(x_0)$.

Let $f:U\rightarrow U$ and $g:V\rightarrow V$ be two maps. $f$ and $g$ are said
to be {\it topologically conjugate} if there exists a homeomorphism $h: U \rightarrow V$ such
that, $h \circ f = g \circ h$. The homeomorphism $h$ is called a {\it topological conjugacy}.
Mappings that are topologically conjugate are completely equivalent in
terms of their dynamics. For example, if $f$ is topologically conjugate to $g$ via
$h$, and $x_0$ is a fixed point for $f$, then $h(x_0)$ is fixed for $g$. Indeed, $h(x_0) = hf(x_0) =
gh(x_0)$.

\section{$(1,2)$-Rational $p$-adic dynamical systems}

In this paper we consider the dynamical system associated with the
$(1,2)$-rational function $f:\C_p\to\C_p$ defined by
\begin{equation}\label{fa}
f(x)=\frac{ax+b}{x^2+cx+d}, \ \ a\neq 0, \ \  a,b,c,d\in \C_p.
\end{equation}
where  $x\neq x_{1,2}=\frac{-c\pm\sqrt{c^2-4d}}{2}$.

We can see that for $(1,2)$-rational function (\ref{fa}) the equation
$f(x)=x$ for fixed points is equivalent to  the equation
\begin{equation}\label{ce}
x^3+cx^2+(d-a)x-b=0.
\end{equation}
Since $\C_p$ is algebraic closed the equation (\ref{ce}) may have
three solutions with one of the following:

(i). One solution having multiplicity three;

(ii). Two solutions, one of which has multiplicity two;

(iii). Three distinct solutions.

In this paper we investigate the behavior of trajectories of an
arbitrary $(1,2)$-rational dynamical system in complex $p$-adic
filed $\C_p$ when there is unique fixed point for $f$, i.e., we
consider the case (i).

The following lemma gives a criterion on parameters of
the function (\ref{fa}) guaranteeing the uniqueness of its fixed point.

\begin{lemma}\label{tp} The function (\ref{fa}) has unique fixed point if and only if
\begin{equation}\label{fc}
   {{-c}\over 3}=-\sqrt{{{d-a}\over 3}}=\sqrt[3]{b} \ \ or \ \ {{-c}\over 3}=\sqrt{{{d-a}\over 3}}=\sqrt[3]{b}.
\end{equation}
\end{lemma}
\begin{proof} {\it Necessariness.} Assume (\ref{fa}) has a unique fixed point, say $x_0$.
Then the LHS of equation (\ref{ce}) (which is equivalent to $f(x)=x$) can be written as
$$x^3+cx^2+(d-a)x-b=(x-x_0)^3.$$
Consequently,
$$\left\{\begin{array}{lll}
3x_0=-c \\[2mm]
3x_0^2=d-a \\[2mm]
x_0^3=b
\end{array}
\right.,$$ which gives
$$x_0={{-c}\over 3}=\pm\sqrt{{{d-a}\over 3}}=\sqrt[3]{b}.$$

{\it Sufficiency.} Assume the coefficients of (\ref{fa})
satisfy (\ref{fc}). Then it can be written as
\begin{equation}\label{fd}
f(x)=\frac{ax-{{c^3}\over {27}}}{x^2+cx+{c^2\over 3}+a},
\ \ a\neq0, \ \  a,c\in \C_p.
\end{equation}
In this case the equation $f(x)=x$ can be written as
$$(x+{{c}\over 3})^3=0.$$
Thus $f(x)$ has unique fixed point $x_0=-{{c}\over 3}$.
\end{proof}

It follows from this lemma that if the function (\ref{fa}) has unique
fixed point then it has the form (\ref{fd}).
Thus we study the dynamical system $(f,\C_p)$ with $f$ given by (\ref{fd}).

Let homeomorphism $h: \C_p \rightarrow \C_p$ is defined by $x=h(t)=t+x_0$. So $h^{-1}(x)=x-x_0$.
Note that, the function $f$ is topologically conjugate to function $h^{-1}\circ f\circ h$.
We have
\begin{equation}\label{fb}
(h^{-1}\circ f\circ h)(t)=\frac{{c\over 3}t^2+(a+{{c^2}\over 9})t}{t^2+{c\over 3}t+a+{{c^2}\over 9}}.
\end{equation}
In (\ref{fb}), the case $c\neq 0$ is studied in \cite{RS2}.

Thus in this paper we consider the case $c=0$ in (\ref{fb}). Therefore, in this paper we study dynamical systems of the following function
\begin{equation}\label{fe}
f(x)=\frac{ax}{x^2+a}, \ \ a\neq 0, \ \ a\in C_p.
\end{equation}
where  $x\neq \hat x_1, \hat x_2 =\pm\sqrt{-a}$.

It is easy to see that function (\ref{fe}) has unique fixed point $x_0=0$.
For (\ref{fe}) we have
$$f'(x_0)=f'(0)=1,$$
i.e.,  the point $x_0$ is an indifferent point for (\ref{fe}).

It follows from (\ref{fa}) that
\begin{equation}
|f(x)-x_0|_p=|f(x)|_p=
|x|_p\cdot \dfrac{|a|_p}{|x^{2}+a|_p}.
\end{equation}
Letting $A=|a|_p$, we have\\
$$|f(x)|_p=|x|_p\cdot\dfrac{A}{|x^{2}+a|_p}.$$
Denote:
\begin{equation}
\mathcal P= \{x\in \mathbb{C}_p:\exists{n}\in\mathbb{N}\cup\{0\}, f^{n}(x)\in\{\hat x_1, \hat x_2\}\}.
\end{equation}

 Let the function $\varphi_{A}:[0,+\infty)\rightarrow[0,+\infty)$ be defined by:
$$\varphi_A(r)=\left\{\begin{array}{ccc}
\ r, \ \ \ \ \  {\rm if }\ \  r < \sqrt{A}\\[2mm]
A^{*} , \ \ \ {\rm if} \ \ r=\sqrt{A}\\[2mm]
\dfrac{A}{r}, \  \  \ \ {\rm if} \ \ r> \sqrt{A}
\end{array} \right.,$$
where $A^*$ is a positive number such that $A^*\geq{\sqrt{A}}$.

\begin{lemma}\label{1}
If $x\in S_r (x_0)$, then the following holds for the function (\ref{fe}):
$$ | f^{n}(x)|_p = \varphi_A^{n} (r).$$
\end{lemma}
We now see that the real dynamical system compiled from $\varphi_A^n$ is directly related to the $p$-adic dynamical system $f^n(x), \ n\geq{1}, \ x\in\mathbb{C}_p\setminus\mathcal P$. The following lemma gives properties to this real dynamical system.

\begin{lemma}\label{2} The function $\varphi_A$ has the following properties
\begin{itemize}
\item[1.] ${\rm Fix}(\varphi_A)=\{{r: 0\leq r < \sqrt{A}}\} \cup \{{\sqrt{A\\}{\rm: if \ } A^{*}=\sqrt{A}}\} $
\item[2.] If $r > \sqrt{A}$ then $\varphi_A^{n} (r)=\frac{A}{r}$  for\ all $n\geq 1$.
\item[3.] If $r=\sqrt{A}$ and $A^{*} > \sqrt {A}$,  then $\varphi_A^{n}(\sqrt{A})=\frac{A}{A^*}$ for all $n\geq{2}$.
\end{itemize}
\end{lemma}

\begin{proof}
\begin{itemize}
\item[1)] This is a simple observation of the equation $\varphi_A(r)=r.$ \\
\item [2)] If $r>\sqrt{A}$, then by definition of function $\varphi_A$, we have
\begin{equation}
\varphi_A(r)=\frac{A}{r}.
\end{equation}
Moreover, $r>\sqrt{A}, \frac{A}{r}<\sqrt{A}$. By part 1 of this lemma, $\frac{A}{r}$ is to be considered a fixed point for $\varphi_A(r)$. Furthermore, $\varphi_A^{n}=\frac{A}{r}$ for all $n\geq 1$.
\item [3.)] The proof of part 3 follows part 2 of this lemma.
\end{itemize}
\end{proof}

From this lemma it follows that
\begin{equation}\label{lim1}
\lim_{n\to\infty}\varphi^{n}\left(r\right)= \left\{\begin{array}{ll|l}
\ \ r, \ \ \ \ \ 0 \leq r < \sqrt{A} \\[2mm]
\sqrt{A}, \ \ \ \ \ r= \sqrt{A}, \ \  A^{*}=\sqrt{A}\\ [2mm]
\ \frac{A}{A^{*}}, \ \ \ \ \ r=\sqrt{A}, A^{*} > \sqrt{A} \\[2mm]
\ \frac{A}{r}, \ \ \ \ \ \ r> \sqrt{A}
\end{array}
\right.,
\end{equation}
for any $r\geq 0$.

Denote:
$$A^{*}(x)= |f(x)|_p,  \  \ {\rm if} \ \ x\in S_{\sqrt{A}}(0).$$

By the applying Lemma \ref{1}, and \ref{2}, and formula (\ref{lim1}) we get the following properties of the $p$-adic dynamical
system complied by the function (\ref{fe}).

\begin{thm}\label{t1}
The $p$-adic dynamical system is generated by the function (\ref{fe}) has the following properties:
\begin{itemize}
\item[1.]
\begin{itemize}
 \item[1.1)]$SI(x_0) = U_{\sqrt{A}}(0)$.
 \item[1.2)]$\mathcal {P} \subset S_{{\sqrt{A}}}(0)$.
 \end{itemize}
 \item[2.] If $r> \sqrt{A}$ and $x \in S_r(0)$, then
$$f^{n}(x) \in{S}_{{A\over r}}(0) {\ \rm for \ all} \ \ n \geq 1.$$
\item[3.] Let $x\in S_{\sqrt{A}} (0)\setminus\mathcal P$.
\begin{itemize}
\item[3.1)] If $A^*(x)=\sqrt{A}$, then
 $$f(x) \in S_{\sqrt{A}}(0).$$
\item[3.2)] If $A^*(x)>\sqrt{A}$, then
 $$f^{n}(x) \in S_{\frac{A}{{A}^{*}(x)}}(0), \ \ \forall \ n \geq 2.$$
\end{itemize}
\end{itemize}
\end{thm}

\begin{proof}
\begin{itemize}
\begin{enumerate}
\item[1.1] By lemma \ref{1} and part 1 of Lemma \ref{2}, sphere $S_r(0)$ is invariant for $f$ if and only if $r<\sqrt{A}$. Consequently, $SI(x_0)=U_{\sqrt{A}}(0)$.
\item[1.2] Note that $|\hat x_1|=|\hat x_2|={\sqrt{A}}$, i.e., $\{\hat x_1, \hat x_2\} \subset S_{\sqrt{A}}(0)$. By part 1.1, and 2 of this theorem if $x \in S_r(0), r\neq\sqrt{A}$, then $f(x) \notin S_{\sqrt{A}}(0)$. By definition of set $\mathcal{P}$, we can conclude that $\mathcal{P} \subset S_{\sqrt{A}}(0)$.
\item[2.]The proof of part 2 easily follows of Lemma \ref{1} and part 2 of Lemma \ref{2}.
\item[3.1] If $x\in S_{\sqrt{A}}(0) \setminus \mathcal {P}$ and $A^*=\sqrt{A}$, we have $|f(x)|_p={\sqrt{A}}$, i.e., $f(x) \in S_{\sqrt{A}}(0)$.
\item[3.2] If $A^* > {\sqrt{A}}$, then by part 2 of this theorem we have $f^n(x) \in S_{\frac{A}{{A^*(x)}}(0)}$, for all $n\geq 2.$
\end{enumerate}
\end{itemize}
\end{proof}

\section{Ergodicity properties of the dynamical system $f(x)=\frac{ax}{x^2+a}$ in $\mathbb {Q}_p.$}
In this section we assume that $\sqrt{-a}$ exists in $\mathbb {Q}_p$. Consider the dynamical system (\ref{fe}) in $\mathbb {Q}_p$.

Define the following set:
$$\Delta = \{{r: 0 < r< \sqrt{A}\}}.$$

From previous section we have
\begin{cor}
The sphere $S_r(0)$ is invariant for $f$ if and only if $r \in \Delta$
\end{cor}

In this section we study ergodicity of dynamical system on each invariant sphere.

\begin{lemma}\label{3}
For every closed ball $V_\rho(c) \subset S_r(0), r\in \Delta$, the following is sufficient to say\\
\begin{equation}
f((V_\rho(c)) = V_\rho(f(c)).
\end{equation}
\end{lemma}

\begin{proof}
By inclusion of $V_\rho(c) \subset S_r(0)$, we have $|c|_p=r$.

Let $x\in V_\rho(c), i.e. \ |x-c|_p \leq \rho$, then
\begin{equation}\label{rr}
|f(x)-f(c)|_p=|x-c|_p \cdot  \frac{|a|_p \cdot |a-xc|_p}{(x^2+a)(c^2+a)}.
\end{equation}
$|x\cdot c|_p= r^2$, since $x\in V_\rho (c) \subset S_r(0)$. Moreover, $|a| _p=A$. If $r \in \Delta$ i.e. $r< \sqrt{A}$, then $|x^2|_p=|x\cdot c|_p = r^2 < A = |a|_p$. By the equality of (\ref{rr}), $|f(x)-f(c)|_p=|x-c|_p \leq \rho$.
\end{proof}

\begin{lemma}\label{rhor}
If $c \in S_r(0)$, $r\in \Delta$, then $|f(c)-c|_p=\frac{r^3}{A}$.
\end{lemma}
\begin{proof}
This follows from the following equality,
\begin{equation}
|f(c)-c|_p=|\frac{-c^3}{c^2+a}|_p=\frac{|c^3|_p}{|c^2+a|_p}=
\frac{r^3}{A}, \ {\rm because} \ r^2 < A.
\end{equation}
\end{proof}

By Lemma \ref{rhor}, we have that $|f(c)-c|_p$ relies on $r$, but does not rely on $c\in S_r(0)$, therefore we define $\rho(r)=|f(c)-c|_p$, if $c \in S_r(0)$.

The following theorems and respective proofs can be reviewed upon in the references as they are similar to the results of paper \cite{RS2}.
\begin{thm}\label{t2}
If $c \in S_r(0), r\in \Delta$ then,
\begin{itemize}
\item[1.] For any $n\geq 1$ the following equality holds $|f^{n+1}(c)-f^n(c)|_p=\rho(r)$.
\item[2.] $f(V_{\rho(r)}(c))=V_{\rho(r)}(c).$
\item[3.] If for some $\theta > 0$, the ball $V_\theta(c) \subset S_r(0)$ is an invariant for $f$, then $\theta \geq \rho(r).$
\end{itemize}
\end{thm}

For each $r \in \Delta$, let us consider a measurable space $(S_r(0),\mathcal B)$, in this case, $\mathcal B$ is the algebra generated by the closed subsets of $S_r(0)$. Each element of $\mathcal B$ is a union of some ball $V_\rho(c)$.

A measure $\bar \mu : \mathcal B \rightarrow \mathbb{R}$ is considered to be Haar Measure if it is defined by $\bar \mu(V_\rho(c)=\rho$.

Notice that $S_r(a)=V_r(a) / V_\frac{r}{p}(a)$, where $p$ is a prime.

Consider the normalized (probability) Haar measure:
$$ \mu(V_\rho(c))=\frac{\bar \mu (V_\rho(c))}{\bar \mu S_r(0))}= \frac{p \rho}{r(p-1)}. $$

By Lemma \ref{3}, we can conclude that $f$ preserves the measure $\mu$, i.e.,
$$\mu(f(V_\rho(c)))=\mu(V_\rho(c)).$$

The dynamical system,$(X, T, \mu)$, where $T: X \rightarrow X$, is a measure preserving transformation whereas $\mu$ us a probability measure. The dynamical system is \textit {ergodic}, if for every invariant set $V$ we have $\mu(V)=0$ or $\mu(V)=1$.

\begin{thm}\label{t3}
The $p$-adic dynamical system $(S_r(0), f, \mu)$ is not ergodic for all prime $p$ and all $r\in \Delta$.
\end{thm}

\begin{proof}
Consider $(S_r(0), f, \mu)$ a dynamical system where $\mu$-normalizer Haar measure.
Note that $V_{\rho(r)}(c) \subset S_r(0)$ is a minimal invariant ball.
By Lemma \ref{rhor}, we have the following:\\
$$ \mu(V_{\rho(r)}(c))=\frac{p \rho(r)}{r (p-1)}= \frac{p {\frac{ r^3}{A}}}{r (p-1)}= \frac {p r^2}{A \ (p-1)}.$$
If $r\in \Delta$, then $r<\sqrt{A}$, i.e., $pr\leq\sqrt{A}$.
By this inequality, we have $$\mu(V_{\rho(r)}(c))\leq\frac{1}{p(p-1)} < 1,$$ for all prime $p$.
\end{proof}

\section{2-Periodic Points}
In this section, we will be interested in finding periodic points of function (\ref{fe}).
\begin{pro}\label{p1}
If function (\ref{fe}) has $k$-periodic points $\{y_0\to y_1\to . . . \to y_k\to y_0\}$, then $$|y_0|_p=|y_1|_p=...=|y_k|_p\leq\sqrt{A}.$$
\end{pro}
\begin{proof}
Let function (\ref{fe}) have $k$-periodic points $\{y_0\to y_1\to . . . \to y_k\to y_0\}$.

Assume that $|y_i|_p>\sqrt{A}$ for some $i\in\{0,1,...,k\}$. Then by part 2 of Theorem \ref{t1} we have $y_{i+1}=f(y_i)\in U_{\sqrt{A}}(0)$, i.e., $|y_{i+1}|_p=r<\sqrt{A}$ and $S_r(0)$ is invariant sphere of function (\ref{fe}). From this $$|y_{i+1}|_p=|y_{i+2}|_p=...=|y_{k}|_p=|y_{0}|_p=...=|y_{i}|_p<\sqrt{A}.$$
This contradicts our assumption. Consequently $|y_i|_p\leq\sqrt{A}$ for any $i\in\{0,1,...,k\}$.

If $|y_{i}|_p=r<\sqrt{A}$, then $|y_{0}|_p=|y_{1}|_p=...=|y_{k}|_p=r<\sqrt{A}$, because $S_r(0)$ is invariant for all $r<\sqrt{A}$.

If $|y_{i}|_p=r=\sqrt{A}$, then by the above-mentioned results $|y_{i+1}|_p\ngtr\sqrt{A}$ and $|y_{i+1}|_p\nless\sqrt{A}$. Consequently, $|y_{0}|_p=|y_{1}|_p=...=|y_{k}|_p=\sqrt{A}$.
\end{proof}

Let us consider 2-periodic points, i.e. consider the equation
\begin{equation}
g(x)\equiv f(f(x))= \frac{ax^3+a^2x}{x^4+3ax^2+a^2}=x.
\end{equation}
This equation is equivalent to $x^2+2a=0$, hence two solutions, $t_{1,2}= \pm \sqrt{-2a}.$
For these points we have $$|t_1|_p=|t_2|_p=\left\{\begin{array}{ll}
\sqrt{A}, \ \ \ \ \ {\rm if} \ p\neq 2 \\[2mm]
\sqrt{\frac{A}{2}}, \ \ \ \ \ {\rm if}  \ p=2.
\end{array}
\right.$$

It is a coincidence that $g'(t_1)=g'(t_2)=9$, i.e. the value does not rely on the parameter $a$. Therefore we have
$$ |g'(t_1)|_p=|g'(t_2)|_p =
\left\{\begin{array}{ll}
1, \ \ \ \ \ {\rm if} \ p\neq 3 \\[2mm]
\frac{1}{9}, \ \ \ \ \ {\rm if}  \ p=3.
\end{array}
\right. $$

Note that the function $g$ is defined in $\mathbb {C}_p \setminus\left\{\hat x_{1,2}= \pm \sqrt{-a}, \hat{\hat{x}}_{1,2,3,4}=\pm{\sqrt{{(-3\pm\sqrt{5})a}\over 2}}\right\}$.\\

Let us define the following:
$$ \mathcal {P}_2= \{x \in \mathbb{ C}_p : \exists n \in \mathbb{N}, {\rm such \  that \ } f^{n} (x) \in \{\hat x_{1,2}, \, \hat{\hat x}_{1,2,3,4}\}\}. $$

\subsection{Case $p=2$.}
In this case we have $|t_1|_2=|t_2|_2=\sqrt{A\over 2}<\sqrt{A}$ and $|t_2-t_1|_2={\sqrt{A}\over{2\sqrt{2}}}$. By Lemma \ref{rhor} and part 2 of Theorem \ref{t2} we have $$\rho\left(\sqrt{A\over 2}\right)=|f(t_1)-t_1|_2={{\sqrt{A}}\over{2\sqrt{2}}} \ \ \mbox{and} \ \ f(V_{{\sqrt{A}}\over{2\sqrt{2}}}(t_1))=V_{{\sqrt{A}}\over{2\sqrt{2}}}(t_1).$$

In this case each fixed point $t_1, t_2$ of $g$ is an indifferent point
 and is the center of a Siegel disk.

\begin{thm}\label{tg1} If $p=2$ then $f(S_r(t_1)\setminus \mathcal P_2)\subseteq S_r(t_2)$,   $f(S_r(t_2)\setminus \mathcal P_2)\subseteq S_r(t_1)$, for any $0<r\leq{{\sqrt{A}}\over{2\sqrt{2}}}$.
\end{thm}
\begin{proof} We shall use the following equalities:
$$f(t_1)=t_2, \ \ f(t_2)=t_1.$$
Let $t_1=\sqrt{-2a}$ and $x\in S_r(t_1)\setminus\mathcal P_2\subset V_{{\sqrt{A}}\over{2\sqrt{2}}}(t_1)$, i.e., $|x-t_1|_2=r\leq{{\sqrt{A}}\over{2\sqrt{2}}}$. We have
\begin{equation}\label{ee}
|f(x)-t_2|_2=|f(x)-f(t_1)|_2=
r\cdot\left|{{-3a+\sqrt{-2a}(x-t_1)}\over{[x-t_1+\sqrt{-a}(\sqrt{2}-1)][x-t_1+\sqrt{-a}(\sqrt{2}+1)]}}\right|_2.
\end{equation}
In RHS of equality (\ref{ee}) we have $|-3a|_2=A$, $|\sqrt{-2a}(x-t_1)|_2=r\sqrt{A\over 2}<A$ and $|\sqrt{2}-1|_2=|\sqrt{2}+1|_2=1$. So $|f(x)-t_2|_2=|f(x)-f(t_1)|_2=r$, i.e., $f(x)\in S_r(t_2)$.

If $x\in S_r(t_2)\setminus\mathcal P_2\subset V_{{\sqrt{A}}\over{2\sqrt{2}}}(t_1)$, then we have
\begin{equation}\label{ed}
|f(x)-t_1|_2=|f(x)-f(t_2)|_2=
r\cdot\left|{{-3a-\sqrt{-2a}(x-t_2)}\over{[x-t_2-\sqrt{-a}(\sqrt{2}-1)][x-t_2-\sqrt{-a}(\sqrt{2}+1)]}}\right|_2.
\end{equation}
Consequently, $|f(x)-t_1|_2=|f(x)-f(t_2)|_2=r$, i.e., $f(x)\in S_r(t_1)$.
\end{proof}

\subsection{Case $p=3$.}
In this case $|t_1|_3=|t_2|_3=|t_2-t_1|_3=\sqrt{A}$ and each fixed point $t_1, t_2$ of $g$ is an attractive point of $g$.

\begin{thm}\label{tg2} If $p=3$ and $r<\sqrt{A}$, then
\begin{itemize}
\item[a)] For any $x\in S_r(t_1)\setminus \mathcal P_2$,
$\lim_{n\to \infty}f^{2n}(x)=t_1$ and $\lim_{n\to \infty}f^{2n+1}(x)=t_2$.
\item[b)] For any $x\in S_r(t_2)\setminus \mathcal P_2$,
$\lim_{n\to \infty}f^{2n}(x)=t_2$ and $\lim_{n\to \infty}f^{2n+1}(x)=t_1$.
\end{itemize}
\end{thm}

\begin{proof} Let $S_r(t_1)\subset S_{\sqrt{A}}(0)$, $r<\sqrt{A}$ and $x\in S_r(t_1)\setminus\mathcal P_2$, i.e., $|x-t_1|_3=r$. We have
$$
|f(x)-t_2|_3=|f(x)-f(t_1)|_3=
|x-t_1|_3\cdot{{|-3a+\sqrt{-2a}(x-t_1)|_3}\over{|x-t_1+\sqrt{-a}(\sqrt{2}-1)|_3|x-t_1+\sqrt{-a}(\sqrt{2}+1)|_3}}.$$
By this equality
\begin{equation}\label{ee1}
|f(x)-t_2|_3=\phi(r)=\left\{\begin{array}{lll}
{r^2\over {\sqrt{A}}}, \ \ \ \ \mbox{if} \ \ {{\sqrt{A}}\over 3}<r<{\sqrt{A}},\\[2mm]
\leq {{\sqrt{A}}\over 9}, \ \ \mbox{if} \ \ r={{\sqrt{A}}\over 3},\\[2mm]
{r\over 3}, \ \ \ \ \ \mbox{if} \ \ r<{{\sqrt{A}}\over 3}.
\end{array}\right.
\end{equation}
For $f^2(x)$ we have
$$
|f^2(x)-t_1|_3=|f^2(x)-f^2(t_1)|_3$$ $$=
|f(x)-t_2|_3\cdot{{|-3a+\sqrt{-2a}(f(x)-t_2)|_3}\over{|f(x)-t_2+\sqrt{-a}(\sqrt{2}-1)|_3|f(x)-t_2+\sqrt{-a}(\sqrt{2}+1)|_3}}=$$ $$ =\phi(\phi(r))=\left\{\begin{array}{lll}
{(\phi(r))^2\over {\sqrt{A}}}, \ \  \  \mbox{if} \ \ {{\sqrt{A}}\over 3}<\phi(r)<{\sqrt{A}},\\[2mm]
\leq {{\sqrt{A}}\over 9}, \ \ \mbox{if} \ \ \phi(r)={{\sqrt{A}}\over 3},\\[2mm]
{\phi(r)\over 3}, \ \ \ \ \ \mbox{if} \ \ \phi(r)<{{\sqrt{A}}\over 3}.
\end{array}\right.
$$
Iterating this argument we obtain the following  formulas for $x\in S_r(t_1)\setminus \mathcal P_2$:
\begin{equation}\label{ef}
    |f^{2n}(x)-t_1|_3=\phi^{2n}(r), \ \ |f^{2n+1}(x)-t_2|_3=\phi^{2n+1}(r).
\end{equation}
Thus the dynamics of the radius $r$ of the spheres is given by the function $\phi:[0,{\sqrt{A}})\to [0,{\sqrt{A}})$, which is defined in formula (\ref{ee1}). The following properties of $\phi$ are obvious:
\begin{itemize}
  \item[$1$.] The set of fixed points of $\phi(r)$ is Fix$(\phi)=\{0\}$;

   \item[$2$.] The fixed point $r=0$ is attractive with basin of attraction $[0,{\sqrt{A}})$, independently on the value $\phi({{\sqrt{A}}\over 3})\leq {{\sqrt{A}}\over 9}$.
   \end{itemize}
Using (\ref{ef}) it is easy to see that the assertion (a) and (b) follows from property $2$.
      \end{proof}

\subsection{Case $p\geq 5$.}
In this case we have $|t_1|_p=|t_2|_p=|t_2-t_1|_p=\sqrt{A}$ and each fixed point $t_1, t_2$ of $g$ is an indifferent point
 and is the center of a Siegel disk.

\begin{thm}\label{tg3} If $p\geq 5$ then $f(S_r(t_1)\setminus \mathcal P_2)\subseteq S_r(t_2)$,   $f(S_r(t_2)\setminus \mathcal P_2)\subseteq S_r(t_1)$, for any $0<r<\sqrt{A}$.
\end{thm}
\begin{proof}
Let $x\in S_r(t_1)\setminus\mathcal P_2\subset V_{\sqrt{A}}(t_1)$, i.e., $|x-t_1|_p=r<{\sqrt{A}}$.
In RHS of equality (\ref{ee}) we have $|-3a|_p=A$, $|\sqrt{-2a}(x-t_1)|_p=r\sqrt{A}<A$ and $|\sqrt{2}-1|_p=|\sqrt{2}+1|_p=1$. So $|f(x)-t_2|_p=|f(x)-f(t_1)|_p=r$, i.e., $f(x)\in S_r(t_2)$.

If $x\in S_r(t_2)\setminus\mathcal P_2\subset V_{{\sqrt{A}}\over{2\sqrt{2}}}(t_1)$, then we have
 $|f(x)-t_1|_p=|f(x)-f(t_2)|_p=r$, i.e., $f(x)\in S_r(t_1)$.

Consequently, $f(S_r(t_1)\setminus \mathcal P_2)\subseteq S_r(t_2)$ and $f(S_r(t_2)\setminus \mathcal P_2)\subseteq S_r(t_1)$, for any $0<r<\sqrt{A}$.
\end{proof}

\section*{ Acknowledgements}
The third author was supported by the National Science Foundation, grant number NSF HRD 1302873.


\begin{thebibliography}{999}

\bibitem{ARS} S. Albeverio, U.A. Rozikov, I.A. Sattarov. $p$-adic $(2,1)$-rational dynamical systems. {\it Jour. Math. Anal. Appl.} {\bf 398}(2) (2013), 553--566.

\bibitem{Ko} N. Koblitz, {\it  $p$-adic numbers, $p$-adic analysis and zeta-function}
Springer, Berlin, 1977.

\bibitem{RS} U.A. Rozikov, I.A. Sattarov. On a non-linear $p$-adic dynamical system. {\it $p$-Adic Numbers, Ultrametric Analysis and Applications}, {\bf 6}(1) (2014), 53--64.

\bibitem{RS2} U.A. Rozikov, I.A. Sattarov. $p$-adic dynamical systems of $(2,2)$-rational functions with unique fixed
point. {\it Chaos, Solitons and Fractals}, {\bf 105} (2017), 260--270.

\bibitem{S} I.A. Sattarov. $p$-adic $(3,2)$-rational dynamical
systems. {\it $p$-Adic Numbers, Ultrametric Analysis and
Applications}, {\bf 7}(1) (2015), 39--55.

\bibitem{Wal} P.Walters, {\it An introduction to ergodic theory.}
 Springer, Berlin-Heidelberg-New York, (1982).
\end{thebibliography}
\end{document}